\documentclass[12pt]{article}

\usepackage{graphicx}
\usepackage{hyperref}
\usepackage[mathscr]{euscript}

\usepackage{amsmath}
\usepackage{amssymb}
\usepackage{amsthm}
\usepackage{bm}
\usepackage{enumitem}
\allowdisplaybreaks[4]

\usepackage{subcaption}
\DeclareCaptionSubType{figure}

\makeatletter
\renewenvironment{thebibliography}[1]
{\section*{\refname\@mkboth{\refname}{\refname}}%
  \list{\@biblabel{\@arabic\c@enumiv}}%
       {\settowidth\labelwidth{\@biblabel{#1}}%
        \leftmargin\labelwidth
        \advance\leftmargin\labelsep
 \setlength\itemsep{-2pt}%←ここの数値を調整(行間のつまり具合)
 \setlength\baselineskip{11pt}%←ここの数値を調整（追加）(文字の大きさ)
        \@openbib@code
        \usecounter{enumiv}%
        \let\p@enumiv\@empty
        \renewcommand\theenumiv{\@arabic\c@enumiv}}%
  \sloppy
  \clubpenalty4000
  \@clubpenalty\clubpenalty
  \widowpenalty4000%
  \sfcode`\.\@m}
 {\def\@noitemerr
   {\@latex@warning{Empty `thebibliography' environment}}%
  \endlist}
\makeatother

\newtheorem{thm}{Theorem}[section]
\newtheorem{prop}[thm]{Proposition}
\newtheorem{lem}[thm]{Lemma}
\newtheorem{cor}[thm]{Corollary}

\theoremstyle{definition}
\newtheorem{dfn}[thm]{Definition}

\theoremstyle{remark}

\newtheorem*{ntn}{Notation}

\numberwithin{equation}{section}

\newcommand{\arxiv}[1]{\href{http://arxiv.org/abs/#1}{\texttt{arXiv:#1}}}

\newcommand{\curv}{\mathop{\mathrm{curv}}\nolimits}

\newcommand{\HM}{\mathscr{H}_{\Gamma}}
\newcommand{\id}[1]{\mathrm{id}_{#1}}

\newcommand{\rad}{\mathop{\mathrm{rad}}\nolimits}

\newcommand{\supp}{\mathop{\mathrm{supp}}\nolimits}
\newcommand{\TDT}{\mathsf{TDT}^{+}}
\newcommand{\TDTo}{\mathsf{TDT}^{\oplus}}

\begin{document}
\title{The Laplacian on some self-conformal fractals\\
and Weyl's asymptotics for its eigenvalues:\\
A survey of the ergodic-theoretic aspects}
\author{Naotaka Kajino\thanks{This work was supported by JSPS KAKENHI Grant Numbers JP25887038, JP15K17554, JP18K18720.}\\
Graduate School of Science, Kobe University}
\date{}
\maketitle

%%%% アブストラクト
\begin{abstract}
This short survey is aimed at sketching the ergodic-theoretic aspects of the author's recent
studies on Weyl's eigenvalue asymptotics for a \emph{``geometrically canonical'' Laplacian}
defined by the author on some self-conformal circle packing fractals including the classical
\emph{Apollonian gasket}. The main result being surveyed is obtained by applying Kesten's
renewal theorem [\emph{Ann.\ Probab.}\ \textbf{2} (1974), 355--386, Theorem 2] for functionals
of Markov chains on general state spaces and provides an alternative probabilistic proof of
the result by Oh and Shah [\emph{Invent.\ Math.}\ \textbf{187} (2012), 1--35, Corollary 1.8]
on the asymptotic distribution of the circles in the Apollonian gasket.
\end{abstract}

%%%% 本文
%%%%%
%
\section{Introduction}\label{sec:intro}
This short survey concerns the author's recent studies in
\cite{K:MFO2016,K:WeylAG,K:WeylDCGMaskit,K:WeylRSC} on Weyl's eigenvalue asymptotics
for a \emph{``geometrically canonical'' Laplacian} defined by the author on circle packing
fractals which are invariant with respect to certain Kleinian groups (i.e., discrete
groups of M\"{o}bius transformations on $\widehat{\mathbb{C}}:=\mathbb{C}\cup\{\infty\}$),
including the classical \emph{Apollonian gasket} (Figure \ref{fig:AGs}).
Here we focus on sketching the ergodic-theoretic aspects of the proof of the eigenvalue
asymptotics, which in fact serves as an alternative proof, though limited to the cases
of certain specific Kleinian groups, of the result by Oh and Shah in
\cite{OhShah:InventMath2012} on the asymptotic distribution of the circles in circle
packing fractals invariant with respect to a very large class of Kleinian groups;
see \cite{K:WeylSurvAnal} for a survey of the analytic aspects of the author's
``geometrically canonical'' Laplacian. Also we restrict our attention to the case
of the Apollonian gasket for simplicity of the presentation.
%and extensions to other
%circle packing fractals will be only briefly mentioned at the end of this article.
%
\begin{ntn}
\begin{itemize}[label=\textup{(1)},align=left,leftmargin=*,parsep=0pt,itemsep=0pt]
\item[\textup{(1)}]We adopt the convention that $\mathbb{N}:=\{n\in\mathbb{Z}\mid n>0\}$,
	i.e., $0\not\in\mathbb{N}$.
\item[\textup{(2)}]The cardinality (the number of elements) of a set $A$ is denoted by $\#A$.
\item[\textup{(3)}]The closure and boundary of $A\subset\mathbb{C}$ in $\mathbb{C}$
	are denoted by $\overline{A}$ and $\partial A$, respectively.
\item[\textup{(4)}]Let $n\in\mathbb{N}$.
	The Euclidean norm on $\mathbb{R}^{n}$ is denoted by $|\cdot|$ and the operator norm
	of a real $n\times n$ matrix $M$ with respect to $|\cdot|$ is denoted by $\|M\|$.
\end{itemize}
\end{ntn}
\section{The Apollonian gasket and its fractal geometry}\label{sec:AG-geometry}
In this section, we introduce the Apollonian gasket and state its geometric
properties needed for our purpose. The following definition and proposition form
the basis of the construction and further detailed studies of the Apollonian gasket.
\begin{dfn}[tangential disk triple, ideal triangle]\label{dfn:tangential-disk-triple}
\begin{itemize}[label=\textup{(0)},align=left,leftmargin=*,parsep=0pt,itemsep=0pt]
\item[\textup{(0)}]We set $S:=\{1,2,3\}$.
\item[\textup{(1)}]Let $D_{1},D_{2},D_{3}\subset\mathbb{C}$ be either three open disks
	or two open disks and an open half-plane. The triple $\mathscr{D}:=(D_{1},D_{2},D_{3})$
	of such open subsets of $\mathbb{C}$ is called a \emph{tangential disk triple}
	if and only if $\#(\overline{D_{j}}\cap\overline{D_{k}})=1$
	(i.e., $D_{j}$ and $D_{k}$ are \emph{externally} tangent) for any $j,k\in S$
	with $j\not=k$.
%	If $\mathscr{D}$ is a tangential disk triple consisting of
%	three disks, then the open triangle in $\mathbb{C}$ with vertices
%	$\{\cent(D_{j})\}_{j\in S}$ is denoted by $\triangle(\mathscr{D})$.
%	We then set $z^{\mathscr{D}}_{j}:=\cent(D_{j})$ for each $j\in S$ with $D_{j}$ a disk.
\item[\textup{(2)}]Let $\mathscr{D}=(D_{1},D_{2},D_{3})$ be a tangential disk triple.
	The open subset $\mathbb{C}\setminus\bigcup_{j\in S}\overline{D_{j}}$ of
	$\mathbb{C}$ is then easily seen to have a unique bounded connected component,
	which is denoted by $T(\mathscr{D})$ and called the \emph{ideal triangle}
	associated with $\mathscr{D}$.
%	Also for each $(j,k,l)\in\{(1,2,3),(2,3,1),(3,1,2)\}$,
%	the unique element of $\overline{D_{k}}\cap\overline{D_{l}}=(\partial D_{k})\cap(\partial D_{l})$
%	is denoted by $q_{j}(\mathscr{D})$ and called a \emph{vertex} of $T(\mathscr{D})$.
\item[\textup{(3)}]A tangential disk triple $\mathscr{D}=(D_{1},D_{2},D_{3})$ is
	called \emph{positively oriented} if and only if its associated ideal triangle
	$T(\mathscr{D})$ is to the left of $\partial T(\mathscr{D})$
	when $\partial T(\mathscr{D})$ is oriented so as to have
	$\{q_{j}(\mathscr{D})\}_{j=1}^{3}$ in this order, where
	$\{q_{j}(\mathscr{D})\}:=\overline{D_{k}}\cap\overline{D_{l}}$
	for $\{j,k,l\}=S$.
\end{itemize}
\noindent
Finally, we define
$\TDT:=\{\mathscr{D}\mid\textrm{$\mathscr{D}$ is a positively oriented tangential disk triple}\}$ and
$\TDTo:=\{\mathscr{D}\mid\textrm{$\mathscr{D}=(D_{1},D_{2},D_{3})\in\TDT$, $D_{1},D_{2},D_{3}$ are disks}\}$.
\end{dfn}
The following proposition is classical and can be shown
by some elementary (though lengthy) Euclidean-geometric arguments.
We set $\rad(D):=r$ and $\curv(D):=r^{-1}$ for an open disk $D\subset\mathbb{C}$
of radius $r$ and $\curv(D):=0$ for an open half-plane $D\subset\mathbb{C}$.
\begin{prop}\label{prop:circumscribed-inscribed}
Let $\mathscr{D}=(D_{1},D_{2},D_{3})\in\TDT$ and set $\alpha:=\curv(D_{1})$,
$\beta:=\curv(D_{2})$, $\gamma:=\curv(D_{3})$ and
$\kappa:=\kappa(\mathscr{D}):=\sqrt{\beta\gamma+\gamma\alpha+\alpha\beta}$.%
\begin{itemize}[label=\textup{(1)},align=left,leftmargin=*,topsep=0pt,parsep=0pt,itemsep=0pt]
\item[\textup{(1)}]Let $D_{\mathrm{cir}}(\mathscr{D})\subset\mathbb{C}$ denote
	the \emph{circumscribed disk} of $T(\mathscr{D})$, i.e., the unique open disk with
	$\{q_{1}(\mathscr{D}),q_{2}(\mathscr{D}),q_{3}(\mathscr{D})\}
		\subset\partial D_{\mathrm{cir}}(\mathscr{D})$.
	Then $\partial D_{\mathrm{cir}}(\mathscr{D})$ is orthogonal to
	$\partial D_{j}$ for any $j\in S$,
	$\overline{T(\mathscr{D})}\setminus\{q_{1}(\mathscr{D}),q_{2}(\mathscr{D}),q_{3}(\mathscr{D})\}
		\subset D_{\mathrm{cir}}(\mathscr{D})$,
	and $\curv(D_{\mathrm{cir}}(\mathscr{D}))=\kappa$.
\item[\textup{(2)}]There exists a unique \emph{inscribed disk}
	$D_{\mathrm{in}}(\mathscr{D})$ of $T(\mathscr{D})$,
	i.e., a unique open disk $D_{\mathrm{in}}(\mathscr{D})\subset\mathbb{C}$
	such that $D_{\mathrm{in}}(\mathscr{D})\subset T(\mathscr{D})$ and
	$\#(\overline{D_{\mathrm{in}}(\mathscr{D})}\cap\overline{D_{j}})=1$
	for any $j\in S$. Moreover,
	$\curv(D_{\mathrm{in}}(\mathscr{D}))=\alpha+\beta+\gamma+2\kappa$.
\end{itemize}
\end{prop}
The following notation is standard in studying self-similar sets.
\begin{dfn}\label{dfn:word-shift}
\begin{itemize}[label=\textup{(1)},align=left,leftmargin=*,parsep=0pt,itemsep=0pt]
\item[\textup{(1)}]We set $W_{0}:=\{\emptyset\}$,
	where $\emptyset$ is an element called the \emph{empty word}, $W_{m}:=S^{m}$
	for $m\in\mathbb{N}$ and $W_{*}:=\bigcup_{m\in\mathbb{N}\cup\{0\}}W_{m}$.
	For $w\in W_{*}$, the unique $m\in\mathbb{N}\cup\{0\}$ satisfying $w\in W_{m}$
	is denoted by $|w|$ and called the \emph{length} of $w$.
\item[\textup{(2)}]Let $w,v\in W_{*}$, $w=w_{1}\ldots w_{m}$, $v=v_{1}\ldots v_{n}$.
	We define $wv\in W_{*}$ by $wv:=w_{1}\ldots w_{m}v_{1}\ldots v_{n}$
	($w\emptyset:=w$, $\emptyset v:=v$). We also define
	$w^{1}\ldots w^{k}$ for $k\geq 3$ and $w^{1},\ldots,w^{k}\in W_{*}$
	inductively by $w^{1}\ldots w^{k}:=(w^{1}\ldots w^{k-1})w^{k}$.
	For $w\in W_{*}$ and $n\in\mathbb{N}\cup\{0\}$ we set $w^{n}:=w\ldots w\in W_{n|w|}$.
	We write $w\leq v$ if and only if $w=v\tau$ for some $\tau\in W_{*}$,
	and write $w\not\asymp v$ if and only if neither $w\leq v$ nor $v\leq w$ holds.
%	Note that $\Sigma_{w}\cap\Sigma_{v}=\emptyset$ if and only if
%	neither $w\leq v$ nor $v\leq w$.
%\item[\textup{(3)}]We set
%	$\Sigma:=S^{\mathbb{N}}=\{\omega_{1}\omega_{2}\omega_{3}\ldots\mid\textrm{$\omega_{k}\in S$ for any $k\in\mathbb{N}$}\}$,
%	which is equipped with the product topology of the discrete topology on $S$,
%	define the \emph{shift map} $\sigma:\Sigma\to\Sigma$ by
%	$\sigma(\omega_{1}\omega_{2}\omega_{3}\ldots):=\omega_{2}\omega_{3}\omega_{4}\ldots$
%	and inductively $\sigma^{0}:=\id{\Sigma}$ and
%	$\sigma^{n}:=\sigma\circ\sigma^{n-1}$ for $n\in\mathbb{N}$.
%	For $j\in S$ we define $\sigma_{j}:\Sigma\to\Sigma$ by
%	$\sigma_{j}(\omega_{1}\omega_{2}\omega_{3}\ldots):=j\omega_{1}\omega_{2}\omega_{3}\ldots$.
%	For $\omega=\omega_{1}\omega_{2}\omega_{3}\ldots\in\Sigma$ and
%	$m\in\mathbb{N}\cup\{0\}$, we set $[\omega]_{m}:=\omega_{1}\ldots\omega_{m}\in W_{m}$.
%\item[\textup{(4)}]For $w=w_{1}\ldots w_{m}\in W_{*}$, we set
%	$\sigma_{w}:=\sigma_{w_{1}}\circ\cdots\circ\sigma_{w_{m}}$
%	($\sigma_{\emptyset}:=\id{\Sigma}$) and
%	$\Sigma_{w}:=\sigma_{w}(\Sigma)=\{w\omega:=\sigma_{w}(\omega)\mid\omega\in\Sigma\}$,
%	and if $w\not=\emptyset$ then we define $w^{\infty}\in\Sigma$
%	by $w^{\infty}:=www\ldots$ in the natural manner.
\end{itemize}
\end{dfn}
Proposition \ref{prop:circumscribed-inscribed}-(2) enables us to define natural
``contraction maps'' $\Phi_{w}:\TDT\to\TDT$ for each $w\in W_{*}$, which in turn
is used to define the Apollonian gasket $K(\mathscr{D})$ associated with
$\mathscr{D}\in\TDT$, as follows.
\begin{dfn}\label{dfn:TDT-Phi}
We define $\Phi_{1},\Phi_{2},\Phi_{3}:\TDT\to\TDT$ by
$\Phi_{1}(\mathscr{D}):=(D_{\mathrm{in}}(\mathscr{D}),D_{2},D_{3})$,
$\Phi_{2}(\mathscr{D}):=(D_{1},D_{\mathrm{in}}(\mathscr{D}),D_{3})$ and
$\Phi_{3}(\mathscr{D}):=(D_{1},D_{2},D_{\mathrm{in}}(\mathscr{D}))$. We also set
$\Phi_{w}:=\Phi_{w_{m}}\circ\cdots\circ\Phi_{w_{1}}$ ($\Phi_{\emptyset}:=\id{\TDT}$)
and $\mathscr{D}_{w}:=\Phi_{w}(\mathscr{D})$ for $w=w_{1}\ldots w_{m}\in W_{*}$ and
$\mathscr{D}\in\TDT$.
\end{dfn}
\begin{dfn}[Apollonian gasket]\label{dfn:AG-construction}
Let $\mathscr{D}\in\TDT$. We define the \emph{Apollonian gasket}
$K(\mathscr{D})$ associated with $\mathscr{D}$ (see Figure \ref{fig:AGs}) by
\begin{equation}\label{eq:AG-construction}
K(\mathscr{D}):=\overline{T(\mathscr{D})}\setminus\bigcup\nolimits_{w\in W_{*}}D_{\mathrm{in}}(\mathscr{D}_{w})
	=\bigcap\nolimits_{m\in\mathbb{N}}\bigcup\nolimits_{w\in W_{m}}\overline{T(\mathscr{D}_{w})}.
\end{equation}
%where the latter equality follows from the fact that for any $w\in W_{*}$,
%\begin{equation}\label{eq:AG-construction-note}
%\overline{T(\mathscr{D}_{w})}\setminus\bigcup\nolimits_{v\in W_{|w|}}D_{\mathrm{in}}(\mathscr{D}_{v})
%	=\overline{T(\mathscr{D}_{w})}\setminus D_{\mathrm{in}}(\mathscr{D}_{w})
%	=\bigcup\nolimits_{j\in S}\overline{T(\mathscr{D}_{wj})}.
%\end{equation}
%In \eqref{eq:AG-construction}, note that
%$\bigl\{\bigcup_{w\in W_{m}}\overline{T(\mathscr{D}_{w})}\bigr\}_{m\in\mathbb{N}\cup\{0\}}$
%is a decreasing sequence of non-empty compact subsets of $\mathbb{C}$
%by \eqref{eq:AG-construction-note} and also that
%$\{q_{1}(\mathscr{D}),q_{2}(\mathscr{D}),q_{3}(\mathscr{D})\}\subset K(\mathscr{D})$.
\end{dfn}
%
%%%%%
\begin{figure}[t]\centering
\subcaptionbox{Examples without a half-plane\label{fig:AGs-disk}}[.520\linewidth]{%
	\includegraphics[height=100pt]{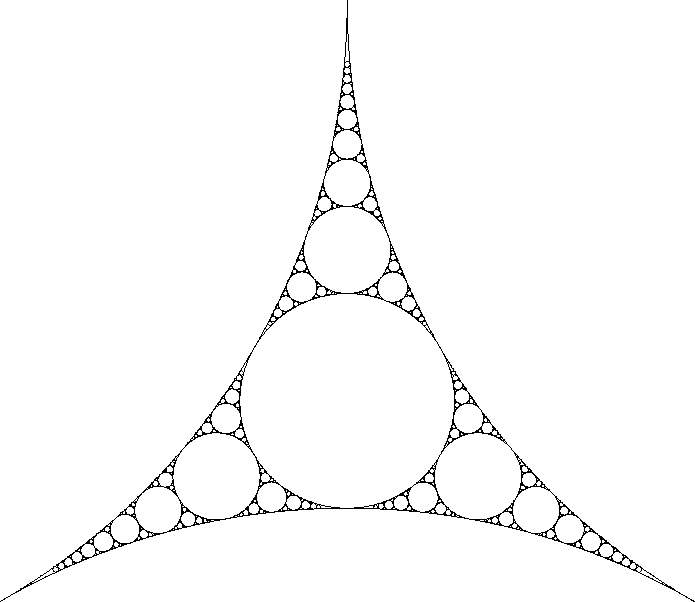}\hspace*{6pt}\includegraphics[height=100pt]{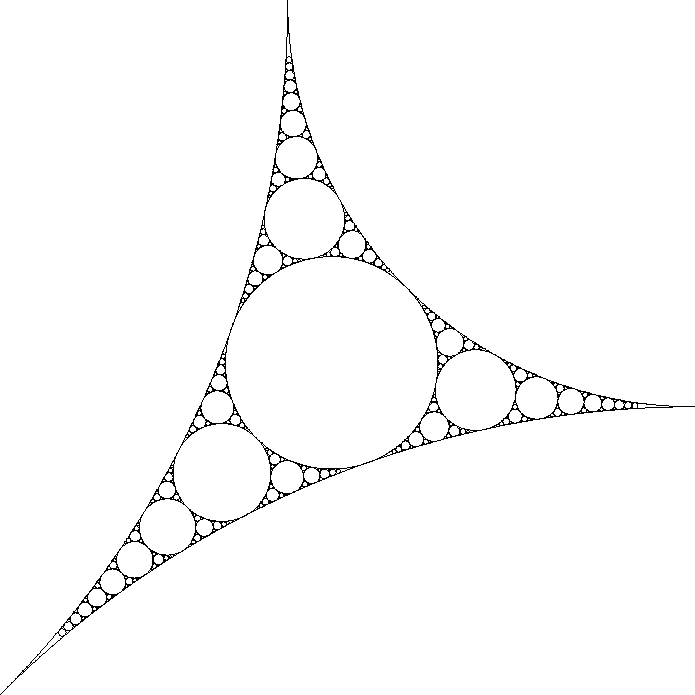}}
\subcaptionbox{Example with a half-plane\label{fig:AG-halfplane}}[.460\linewidth]{%
	\includegraphics[height=100pt]{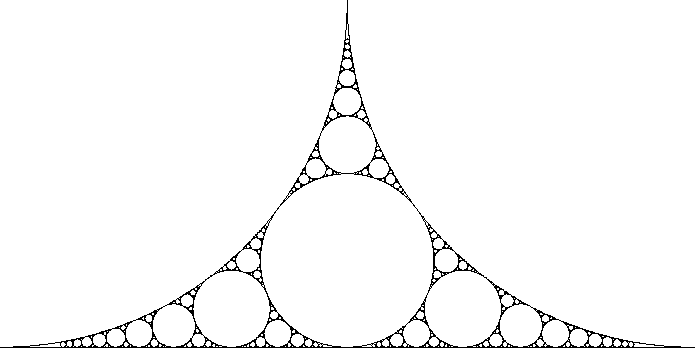}}
\caption{The Apollonian gaskets $K(\mathscr{D})$ associated with $\mathscr{D}\in\TDT$}
\label{fig:AGs}
\end{figure}
%%%%%
%
The curvatures of the disks involved in \eqref{eq:AG-construction}
admit the following simple expression.
\begin{dfn}\label{dfn:Mw}
We define $4\times 4$ real matrices $M_{1},M_{2},M_{3}$ by
\begin{equation}\label{eq:M1M2M3}
M_{1}:=\begin{pmatrix}1 & 0 & 0 & 0\\1 & 1 & 0 & 1\\1 & 0 & 1 & 1\\2 & 0 & 0 & 1\end{pmatrix},
	\mspace{15mu}
	M_{2}:=\begin{pmatrix}1 & 1 & 0 & 1\\0 & 1 & 0 & 0\\0 & 1 & 1 & 1\\0 & 2 & 0 & 1\end{pmatrix},
	\mspace{15mu}
	M_{3}:=\begin{pmatrix}1 & 0 & 1 & 1\\0 & 1 & 1 & 1\\0 & 0 & 1 & 0\\0 & 0 & 2 & 1\end{pmatrix}
\end{equation}
and set $M_{w}:=M_{w_{1}}\cdots M_{w_{m}}$ for each $w=w_{1}\ldots w_{m}\in W_{*}$
\textup{($M_{\emptyset}:=\id{4\times 4}$)}. Note that then for any
$n\in\mathbb{N}\cup\{0\}$ we easily obtain
\begin{gather}\label{eq:M1nM2nM3n}
M_{1^{n}}=\begin{pmatrix}1 & 0 & 0 & 0\\n^{2} & 1 & 0 & n\\n^{2} & 0 & 1 & n\\2n & 0 & 0 & 1\end{pmatrix},
	\mspace{10mu}
	M_{2^{n}}=\begin{pmatrix}1 & n^{2} & 0 & n\\0 & 1 & 0 & 0\\0 & n^{2} & 1 & n\\0 & 2n & 0 & 1\end{pmatrix},
	\mspace{10mu}
	M_{3^{n}}=\begin{pmatrix}1 & 0 & n^{2} & n\\0 & 1 & n^{2} & n\\0 & 0 & 1 & 0\\0 & 0 & 2n & 1\end{pmatrix},\\
\begin{split}
M_{2^{n}3}&=\left(\begin{smallmatrix}1\,&\,n^{2}\,&\,(n+1)^{2}\,&\,n^{2}+n+1\\[1pt]0\,&\,1\,&\,1\,&\,1\\[1pt]0\,&\,n^{2}\,&\,(n+1)^{2}\,&\,n(n+1)\\[1pt]0\,&\,2n\,&\,2(n+1)\,&\,2n+1\end{smallmatrix}\right),\quad
	M_{3^{n}2}=\left(\begin{smallmatrix}1\,&\,(n+1)^{2}\,&\,n^{2}\,&\,n^{2}+n+1\\[1pt]0\,&\,(n+1)^{2}\,&\,n^{2}\,&\,n(n+1)\\[1pt]0\,&\,1\,&\,1\,&\,1\\[1pt]0\,&\,2(n+1)\,&\,2n\,&\,2n+1\end{smallmatrix}\right),\\
M_{3^{n}1}&=\left(\begin{smallmatrix}(n+1)^{2}\,&\,0\,&\,n^{2}\,&\,n(n+1)\\[1pt](n+1)^{2}\,&\,1\,&\,n^{2}\,&\,n^{2}+n+1\\[1pt]1\,&\,0\,&\,1\,&\,1\\[1pt]2(n+1)\,&\,0\,&\,2n\,&\,2n+1\end{smallmatrix}\right),\quad
	M_{1^{n}3}=\left(\begin{smallmatrix}1\,&\,0\,&\,1\,&\,1\\[1pt]n^{2}\,&\,1\,&\,(n+1)^{2}\,&\,n^{2}+n+1\\[1pt]n^{2}\,&\,0\,&\,(n+1)^{2}\,&\,n(n+1)\\[1pt]2n\,&\,0\,&\,2(n+1)\,&\,2n+1\end{smallmatrix}\right),\\
M_{1^{n}2}&=\left(\begin{smallmatrix}1\,&\,1\,&\,0\,&\,1\\[1pt]n^{2}\,&\,(n+1)^{2}\,&\,0\,&\,n(n+1)\\[1pt]n^{2}\,&\,(n+1)^{2}\,&\,1\,&\,n^{2}+n+1\\[1pt]2n\,&\,2(n+1)\,&\,0\,&\,2n+1\end{smallmatrix}\right),\quad
	M_{2^{n}1}=\left(\begin{smallmatrix}(n+1)^{2}\,&\,n^{2}\,&\,0\,&\,n(n+1)\\[1pt]1\,&\,1\,&\,0\,&\,1\\[1pt](n+1)^{2}\,&\,n^{2}\,&\,1\,&\,n^{2}+n+1\\[1pt]2(n+1)\,&\,2n\,&\,0\,&\,2n+1\end{smallmatrix}\right).
\end{split}
\label{eq:M1n23M2n31M3n12}
\end{gather}
\end{dfn}
\begin{prop}\label{prop:curvatures-Mw}
Let $\mathscr{D}=(D_{1},D_{2},D_{3})\in\TDT$, let $\alpha,\beta,\gamma,\kappa$
be as in Proposition \textup{\ref{prop:circumscribed-inscribed}}, let $w\in W_{*}$
and set $(D_{w,1},D_{w,2},D_{w,3}):=\mathscr{D}_{w}$. Then
\begin{equation}
\bigl(\curv(D_{w,1}),\curv(D_{w,2}),\curv(D_{w,3}),\kappa(\mathscr{D}_{w})\bigr)
	=(\alpha,\beta,\gamma,\kappa)M_{w}.
\label{eq:curvatures-Mw}
\end{equation}
\end{prop}
\begin{proof}
This follows by induction in $|w|$ using
Proposition \ref{prop:circumscribed-inscribed}-(2) and Definition \ref{dfn:TDT-Phi}.
\end{proof}
We next collect basic facts regarding the Hausdorff dimension and measure of $K(\mathscr{D})$.
For each $s\in(0,+\infty)$ let $\mathscr{H}^{s}:2^{\mathbb{C}}\to[0,+\infty]$ denote
the $s$-dimensional Hausdorff (outer) measure on $\mathbb{C}$ with respect to the
Euclidean metric, and for each $A\subset\mathbb{C}$ let
$\dim_{\mathrm{H}}A:=\sup\{s\in(0,+\infty)\mid\mathscr{H}^{s}(A)=+\infty\}
	=\inf\{s\in(0,+\infty)\mid\mathscr{H}^{s}(A)=0\}$
denote its Hausdorff dimension. As is well known, it easily follows from the definition
of $\mathscr{H}^{s}$ that the image $f(A)$ of $A\subset\mathbb{C}$ by a Lipschitz
continuous map $f:A\to\mathbb{C}$ with Lipschitz constant $C\in[0,+\infty)$ satisfies
$\mathscr{H}^{s}(f(A))\leq C^{s}\mathscr{H}^{s}(A)$ for any $s\in(0,+\infty)$.
On the basis of this observation, we easily get the following lemma.
\begin{lem}\label{lem:AG-Hausdorff-dim}
Let $\mathscr{D},\mathscr{D}'\in\TDT$. Then there exists $c\in(0,+\infty)$ such that
$\mathscr{H}^{s}(K(\mathscr{D}))\leq c^{s}\mathscr{H}^{s}(K(\mathscr{D}'))$
for any $s\in(0,+\infty)$. In particular,
$\dim_{\mathrm{H}}K(\mathscr{D})=\dim_{\mathrm{H}}K(\mathscr{D}')$.
\end{lem}
\begin{proof}
Let $f_{\mathscr{D}',\mathscr{D}}$ denote the unique orientation-preserving M\"{o}bius
transformation on the Riemann sphere $\widehat{\mathbb{C}}:=\mathbb{C}\cup\{\infty\}$
such that $f_{\mathscr{D}',\mathscr{D}}(q_{j}(\mathscr{D}'))=q_{j}(\mathscr{D})$ for any
$j\in S$. Then $f_{\mathscr{D}',\mathscr{D}}(K(\mathscr{D}'))=K(\mathscr{D})$,
since M\"{o}bius transformations map any open disk in $\widehat{\mathbb{C}}$
onto another. Now the assertion follows from the Lipschitz continuity of
$f_{\mathscr{D}',\mathscr{D}}|_{\overline{D_{\mathrm{cir}}(\mathscr{D}')}}$.
\end{proof}
\begin{dfn}\label{dfn:AG-Hausdorff-dim}
Noting Lemma \ref{lem:AG-Hausdorff-dim} and choosing $\mathscr{D}\in\TDT$
arbitrarily, we define
\begin{equation}\label{eq:AG-Hausdorff-dim}
d:=\dim_{\mathrm{H}}K(\mathscr{D}).
\end{equation}
\end{dfn}
\begin{thm}[Boyd \cite{Boyd:Mathematika1973}; see also \cite{Hirst:JLMS1967,MauldinUrbanski:AdvMath1998,McMullen:AmerJMath1998}]\label{thm:AG-Hausdorff-dim-bounds}
$1<d<2$. In fact,
\begin{equation}\label{eq:AG-Hausdorff-dim-bounds}
1.300197<d<1.314534.
\end{equation}
\end{thm}
Moreover, for the $d$-dimensional Hausdorff measure 
$\mathscr{H}^{d}(K(\mathscr{D}))$ of $K(\mathscr{D})$
we have the following theorem, which was proved first by Sullivan
\cite{Sullivan:Acta1984} through considerations on the isometric action of
M\"{o}bius transformations on the three-dimensional hyperbolic space, and 
later by Mauldin and Urba\'{n}ski \cite{MauldinUrbanski:AdvMath1998}
through purely two-dimensional arguments.
\begin{thm}[{Sullivan \cite[Theorem 2]{Sullivan:Acta1984}, Mauldin and Urba\'{n}ski \cite[Theorem 2.6]{MauldinUrbanski:AdvMath1998}}]\label{thm:AG-Hausdorff-meas}
For any $\mathscr{D}\in\TDT$,
\begin{equation}\label{eq:AG-Hausdorff-meas}
0<\mathscr{H}^{d}(K(\mathscr{D}))<+\infty.
\end{equation}
\end{thm}
Note that for each
$\mathscr{D}=(D_{1},D_{2},D_{3}),\mathscr{D}'=(D'_{1},D'_{2},D'_{3})\in\TDT$,
the M\"{o}bius transformation $f_{\mathscr{D}',\mathscr{D}}$ as in the proof of
Lemma \ref{lem:AG-Hausdorff-dim} is a Euclidean isometry if and only if
$\curv(D_{j})=\curv(D'_{j})$ for any $j\in S$, and in particular that the value of
$\mathscr{H}^{d}(K(\mathscr{D}))$ for $\mathscr{D}=(D_{1},D_{2},D_{3})\in\TDT$
is determined solely by $(\curv(D_{j}))_{j\in S}$. Thus we can use the triple
$(\curv(D_{j}))_{j\in S}$ as a parameter to represent the particular Euclidean
geometry of \emph{each} Apollonian gasket $K(\mathscr{D})$, and the following
definition introduces the space of such parameters, regarded as that of all
possible Euclidean geometries of $K(\mathscr{D})$, $\mathscr{D}\in\TDT$.
\begin{dfn}\label{dfn:AG-geometry-parametrize}
\begin{itemize}[label=\textup{(0)},align=left,leftmargin=*,parsep=0pt,itemsep=0pt]
\item[\textup{(0)}]We set
	$\kappa(g):=\sqrt{\beta\gamma+\gamma\alpha+\alpha\beta}$
	for $g=(\alpha,\beta,\gamma)\in[0,+\infty)^{3}$.
\item[\textup{(1)}]We define $\Gamma\subset[0,+\infty)^{4}$ by
	$\Gamma:=\bigl\{(g,\kappa(g))\bigm|\textrm{$g\in[0,+\infty)^{3}$, $\kappa(g)>0$}\bigr\}$
	and equip $\Gamma$ with the metric $\rho_{\Gamma}:\Gamma\times\Gamma\to[0,+\infty)$ given by
	$\rho_{\Gamma}\bigl((g_{1},\kappa(g_{1})),(g_{2},\kappa(g_{2}))\bigr):=|g_{1}-g_{2}|$.
	We also set $\Gamma^{\circ}:=\Gamma\cap(0,+\infty)^{4}$, which is an open subset of $\Gamma$.
\item[\textup{(2)}]For each $\varepsilon\in(0,1)$, we define a compact subset
	$\Gamma_{\varepsilon}$ of $(\Gamma,\rho_{\Gamma})$ by
	\begin{equation}\label{eq:AG-geometry-parametrize-Gamma-epsilon}
	\Gamma_{\varepsilon}:=\bigl\{(g,\kappa(g))\bigm|
		\textrm{$g=(\alpha,\beta,\gamma)\in[0,\varepsilon^{-1}]^{3}$, $\min\{\beta+\gamma,\gamma+\alpha,\alpha+\beta\}\geq\varepsilon$}\bigr\}.
	\end{equation}
\end{itemize}
\end{dfn}
In Definition \ref{dfn:AG-geometry-parametrize}, the coordinate $\kappa(g)$ is attached
for the sake of the simplicity in applying Proposition \ref{prop:curvatures-Mw}.
It is of course redundant for parametrizing the geometries of $K(\mathscr{D})$
and each $(g,\kappa(g))\in\Gamma$ should (and will) be identified with $g$, which is
why the metric $\rho_{\Gamma}$ has been defined as above. Clearly
$\bigl(\curv(D_{1}),\curv(D_{2}),\curv(D_{3}),\kappa(\mathscr{D})\bigr)\in\Gamma$
for any $\mathscr{D}=(D_{1},D_{2},D_{3})\in\TDT$, and conversely it is not difficult to see that
any $(g,\kappa(g))\in\Gamma$ is of this form for some $\mathscr{D}=(D_{1},D_{2},D_{3})\in\TDT$.
Thus the value of the Hausdorff measure $\mathscr{H}^{d}(K(\mathscr{D}))$
of $K(\mathscr{D})$ for $\mathscr{D}\in\TDT$ induces a function
$\HM:\Gamma\to(0,+\infty)$ on the parameter space $\Gamma$ defined as follows.
\begin{dfn}\label{dfn:AG-HM}
Recalling Theorem \textup{\ref{thm:AG-Hausdorff-meas}}, we define $\HM:\Gamma\to(0,+\infty)$ by
\begin{equation}\label{eq:AG-HM}
\HM(g):=\HM(\alpha,\beta,\gamma):=\mathscr{H}^{d}(K(\mathscr{D}))
	\qquad\textrm{for each $g=(\alpha,\beta,\gamma,\kappa)\in\Gamma$,}
\end{equation}
where we take any $\mathscr{D}=(D_{1},D_{2},D_{3})\in\TDT$ with
$\bigl(\curv(D_{1}),\curv(D_{2}),\curv(D_{3})\bigr)=(\alpha,\beta,\gamma)$.
Note that it is easy to see that for any $(\alpha,\beta,\gamma,\kappa)\in\Gamma$
and any $s\in(0,+\infty)$,
\begin{equation}\label{eq:AG-HM-invariant}
\HM(\alpha,\beta,\gamma)
	=\HM(\beta,\gamma,\alpha)
	=\HM(\alpha,\gamma,\beta)
	=s^{d}\HM(s\alpha,s\beta,s\gamma).
\end{equation}
\end{dfn}
For our purpose we will need some uniform continuity of $\HM$. In fact,
we can prove the following theorem by showing that the M\"{o}bius transformation
$f_{\mathscr{D}',\mathscr{D}}$ as in the proof of Lemma \ref{lem:AG-Hausdorff-dim} is
close to a Euclidean isometry uniformly on $\overline{D_{\mathrm{cir}}(\mathscr{D}')}$
if $\mathscr{D},\mathscr{D}'\in\TDT$ are close in the parameter space $\Gamma$.
Recall Proposition \ref{prop:curvatures-Mw}, which yields
$gM_{w}\in\Gamma$ and $\HM(gM_{w})=\mathscr{H}^{d}(K_{w}(\mathscr{D}))$
for $g,\mathscr{D}$ as in \eqref{eq:AG-HM} and any $w\in W_{*}$.
\begin{thm}\label{thm:AG-HM-Lipschitz}
There exist $c_{1},c_{2}\in(0,+\infty)$ such that for each $\varepsilon\in(0,1)$,
\begin{align}\label{eq:AG-HM-Lipschitz}
\bigl|\HM(g_{1}M_{w})-\HM(g_{2}M_{w})\bigr|
	&\leq c_{1,\varepsilon}\HM(g_{1}M_{w})\rho_{\Gamma}(g_{1},g_{2})\\
\textrm{and}\qquad\HM(g_{2}M_{w})&\leq c_{2,\varepsilon}\HM(g_{1}M_{w})
\label{eq:AG-HM-comparable}
\end{align}
for any $g_{1},g_{2}\in\Gamma_{\varepsilon}$ and any $w\in W_{*}$,
where $c_{1,\varepsilon}:=c_{1}\varepsilon^{1-30d}$ and
$c_{2,\varepsilon}:=c_{2}\varepsilon^{-15d}$.
\end{thm}
Theorem \ref{thm:AG-HM-Lipschitz} easily leads to the following corollary.
Note that for each $w\in W_{*}$, $g\mapsto\HM(gM_{w})^{1/d}gM_{w}$
defines a map from $\{g\in\Gamma\mid\HM(g)=1\}$ to itself
by virtue of \eqref{eq:AG-HM-invariant}, and \eqref{eq:AG-HM-Lipschitz-Mw}
below gives an upper bound on its Lipschitz constant.
\begin{cor}\label{cor:AG-HM-Lipschitz}
There exist $c_{3},c_{4}\in(0,+\infty)$ such that for each $\varepsilon\in(0,1)$,
\begin{align}\label{eq:AG-HM-Lipschitz-Mw}
\rho_{\Gamma}\bigl(\HM(g_{1}M_{w})^{1/d}g_{1}M_{w},\HM(g_{2}M_{w})^{1/d}g_{2}M_{w}\bigr)
	&\leq c_{3,\varepsilon}\HM\bigl(\tfrac{g_{0}M_{w}}{\|M_{w}\|}\bigr)^{1/d}\rho_{\Gamma}(g_{1},g_{2})\!\\
\textrm{and}\qquad\bigl|\log\HM(g_{1}M_{w})-\log\HM(g_{2}M_{w})\bigr|
	&\leq c_{4,\varepsilon}\rho_{\Gamma}(g_{1},g_{2})
\label{eq:AG-HM-Lipschitz-log}
\end{align}
for any $g_{0},g_{1},g_{2}\in\Gamma_{\varepsilon}$ and any $w\in W_{*}$,
where $c_{3,\varepsilon}:=c_{3}\varepsilon^{-45d}$ and
$c_{4,\varepsilon}:=c_{4}\varepsilon^{1-45d}$.
\end{cor}
The factor $\HM\bigl(\tfrac{g_{0}M_{w}}{\|M_{w}\|}\bigr)^{1/d}$
in the right-hand side of \eqref{eq:AG-HM-Lipschitz-Mw} cannot be bounded
by a constant independent of $w$; in fact, it can be shown that
$\inf_{n\in\mathbb{N}}\HM\bigl(\tfrac{g_{0}M_{j^{n}}}{\|M_{j^{n}}\|}\bigr)/n>0$
for any $g_{0}\in\Gamma$ and any $j\in S$. This factor in \eqref{eq:AG-HM-Lipschitz-Mw},
however, can be replaced by a constant for the words $w\in W_{*}$ ending with $j^{n}k$
for some $j,k\in S$ with $j\not=k$ and $n\in\mathbb{N}$, and consequently we obtain
Proposition \ref{prop:AG-HM-I-comparable-norm} and Corollary \ref{cor:AG-HM-Lipschitz-I}
below. The idea of working with this class of words to study dynamics on the Apollonian
gasket is originally due to Mauldin and Urba\'{n}ski \cite{MauldinUrbanski:AdvMath1998}
and plays crucial roles also in our study.
\begin{dfn}\label{dfn:index-I}
We define $I\subset W_{*}\setminus\{\emptyset\}$ by
$I:=\{j^{n}k\mid\textrm{$j,k\in S$, $j\not=k$, $n\in\mathbb{N}$}\}$, so that
$\tau\not\asymp\upsilon$ for any $\tau,\upsilon\in I$ with $\tau\not=\upsilon$ and
$K(\mathscr{D})\setminus V_{0}(\mathscr{D})=\bigcup_{\tau\in I}K_{\tau}(\mathscr{D})$
for $\mathscr{D}\in\TDT$.
\end{dfn}
\begin{prop}\label{prop:AG-HM-I-comparable-norm}
There exist $c_{5},c_{6}\in(0,+\infty)$ and $\varepsilon_{0}\in(0,1)$
such that for any $g\in\Gamma$ and any $\tau\in I$,
\begin{align}\label{eq:AG-HM-I-comparable-norm}
c_{5}&|gM_{\tau}|^{-d}\leq\HM(gM_{\tau})\leq c_{6}|gM_{\tau}|^{-d}\\
\textrm{and}&\qquad\HM(gM_{\tau})^{1/d}gM_{\tau}\in\Gamma_{\varepsilon_{0}}.
\label{eq:AG-HM-I-Gamma-epsilon0}
\end{align}
\end{prop}
\begin{cor}\label{cor:AG-HM-Lipschitz-I}
There exists $c_{7}\in(0,+\infty)$ such that for each $\varepsilon\in(0,1)$,
\begin{equation}\label{eq:AG-HM-Lipschitz-I-Mw}
\rho_{\Gamma}\bigl(\HM(g_{1}M_{w\tau})^{1/d}g_{1}M_{w\tau},\HM(g_{2}M_{w\tau})^{1/d}g_{2}M_{w\tau}\bigr)
	\leq c_{7,\varepsilon}\rho_{\Gamma}(g_{1},g_{2})
\end{equation}
for any $g_{1},g_{2}\in\Gamma_{\varepsilon}$, any $w\in W_{*}$ and any $\tau\in I$,
where $c_{7,\varepsilon}:=c_{7}\varepsilon^{-45d}$.
\end{cor}
\section{The asymptotic formula for counting functions}\label{sec:counting}
The following theorem is a corollary of the general result by Oh and Shah
\cite[Theorem 1.4]{OhShah:InventMath2012} for circle packing fractals
invariant with respect to a large class of Kleinian groups.
\begin{thm}[{Oh and Shah \cite[Corollary 1.8]{OhShah:InventMath2012}}]\label{thm:circle-counting-AG}
There exists $c_{8}\in(0,+\infty)$ such that for any $\mathscr{D}\in\TDT$,
\begin{equation}\label{eq:Oh-Shah-circle-counting}
\lim\nolimits_{\lambda\to+\infty}\lambda^{-d}\#\{w\in W_{*}\mid\curv(D_{\mathrm{in}}(\mathscr{D}_{w}))\leq\lambda\}
	=c_{8}\mathscr{H}^{d}(K(\mathscr{D})).
\end{equation}
\end{thm}
Theorem \ref{thm:circle-counting-AG} can be deduced from the following theorem
applicable to more general counting functions, including that for the
eigenvalues of the ``geometrically canonical'' Laplacian on the Apollonian gasket
introduced by Teplyaev in \cite[Theorem 5.17]{Tep:energySG} and studied extensively
by the author in \cite{K:WeylAG}; see also \cite[\S3 and \S4]{K:WeylSurvAnal}
for a short exposition.
\begin{thm}[\cite{K:WeylAG}]\label{thm:general-counting-AG}
Let $\Gamma'$ denote either of $\Gamma$ and $\Gamma^{\circ}$, and
for each $n\in\mathbb{N}$ let $\lambda_{n}:\Gamma'\to(0,+\infty)$ be continuous and
satisfy $\lambda_{n}(sg)=s\lambda_{n}(g)$ for any $(g,s)\in\Gamma'\times(0,\infty)$.
Suppose that $\lambda_{1}(g)=\min_{n\in\mathbb{N}}\lambda_{n}(g)$ and
$\lim_{n\to\infty}\lambda_{n}(g)=+\infty$ for any $g\in\Gamma'$, set
$\mathscr{N}(g,\lambda):=\#\{n\in\mathbb{N}\mid\lambda_{n}(g)\leq\lambda\}$
for each $(g,\lambda)\in\Gamma'\times[0,+\infty)$, and suppose that there exist
$\eta\in[0,d)$ and $c\in(0,+\infty)$ such that for any
$(g,\lambda)\in(\Gamma'\cap\Gamma_{\varepsilon_{0}})\times[0,+\infty)$,
\begin{equation}\label{eq:general-counting-AG-remainder}
\sum\nolimits_{\tau\in I}\mathscr{N}(gM_{\tau},\lambda)\leq\mathscr{N}(g,\lambda)
	\leq\sum\nolimits_{\tau\in I}\mathscr{N}(gM_{\tau},\lambda)+c\lambda^{\eta}+c,
\end{equation}
where $\varepsilon_{0}$ is as in \eqref{eq:AG-HM-I-Gamma-epsilon0}. Then there exists
$c_{0}\in(0,+\infty)$ such that for any $g\in\Gamma'\cap\Gamma_{\varepsilon_{0}}$,%
\begin{equation}\label{eq:general-counting-AG}
\lim\nolimits_{\lambda\to+\infty}\lambda^{-d}\mathscr{N}(g,\lambda)=c_{0}\HM(g).
\end{equation}
Moreover, additionally if for each $g\in\Gamma'$ there exists $c_{g}\in(0,+\infty)$ such
that \eqref{eq:general-counting-AG-remainder} with $c_{g}$ in place of $c$ holds for any
$\lambda\in[0,+\infty)$, then \eqref{eq:general-counting-AG} holds for any $g\in\Gamma'$.\vspace*{-3.05pt}\newpage
\end{thm}
Theorem \ref{thm:circle-counting-AG} is an easy corollary of
Theorem \ref{thm:general-counting-AG}. Indeed, choose a bijection
$\mathbb{N}\ni n\mapsto w(n)\in W_{*}$ with $w(1)=\emptyset$ and define
$\lambda_{n}(g):=gM_{w(n)}\Bigl(\begin{smallmatrix}1\\1\\1\\2\end{smallmatrix}\Bigr)$
for each $(g,n)\in\Gamma\times\mathbb{N}$. Then
$\#\{w\in W_{*}\mid\curv(D_{\mathrm{in}}(\mathscr{D}_{w}))\leq\lambda\}
	=\#\{n\in\mathbb{N}\mid\lambda_{n}(g(\mathscr{D}))\leq\lambda\}
	=:\mathscr{N}(g(\mathscr{D}),\lambda)$
for any $\lambda\in[0,+\infty)$ for each $\mathscr{D}=(D_{1},D_{2},D_{3})\in\TDT$ by
Proposition \ref{prop:curvatures-Mw} and Proposition \ref{prop:circumscribed-inscribed}-(2), where
$g(\mathscr{D}):=\bigl(\curv(D_{1}),\curv(D_{2}),\curv(D_{3}),\kappa(\mathscr{D})\bigr)\in\Gamma$,
and clearly $\lambda_{1}(g)=\min_{n\in\mathbb{N}}\lambda_{n}(g)$ and
$\lim_{n\to\infty}\lambda_{n}(g)=+\infty$ for any $g\in\Gamma$ by Definition \ref{dfn:Mw}.
Furthermore for each $g=(\alpha,\beta,\gamma,\kappa)\in\Gamma$ and for any
$\lambda\in[0,+\infty)$, taking $\mathscr{D}\in\TDT$ with $g(\mathscr{D})=g$ and noting that
$W_{*}=\{j^{n}\mid\textrm{$j\in S$, $n\in\mathbb{N}\cup\{0\}$}\}\cup\bigcup_{\tau\in I}\{\tau w\mid w\in W_{*}\}$
with the union disjoint, from \eqref{eq:M1nM2nM3n} we easily obtain
\begin{align*}
\mathscr{N}(g,\lambda)
	&=\sum\nolimits_{\tau\in I}\mathscr{N}(gM_{\tau},\lambda)
	+\#\biggl\{j^{n}\biggm|\textrm{$j\in S$, $n\in\mathbb{N}\cup\{0\}$, $gM_{j^{n}}\Bigl(\begin{smallmatrix}1\\1\\1\\2\end{smallmatrix}\Bigr)\leq\lambda$}\biggr\}\\
&\leq\sum\nolimits_{\tau\in I}\mathscr{N}(gM_{\tau},\lambda)
	+3\bigl(\min\{\beta+\gamma,\gamma+\alpha,\alpha+\beta\}\bigr)^{-1/2}\lambda^{1/2}.
\end{align*}
Thus Theorem \ref{thm:general-counting-AG} is applicable to the above choice of $\lambda_{n}(g)$
and implies that \eqref{eq:general-counting-AG} holds for any $g\in\Gamma$ for some
$c_{0}\in(0,+\infty)$, which is nothing but Theorem \ref{thm:circle-counting-AG}.

The rest of this article is devoted to a brief sketch of the proof of Theorem \ref{thm:general-counting-AG}.
The key idea is to apply Kesten's renewal theorem \cite[Theorem 2]{Kesten:AOP1974}
to the Markov chain on $\{g\in\Gamma\mid\HM(g)=1\}$ defined as follows.
\begin{dfn}\label{dfn:words-I}
\begin{itemize}[label=\textup{(1)},align=left,leftmargin=*,parsep=0pt,itemsep=0pt]
\item[\textup{(1)}]We set $W^{I}_{0}:=\{\emptyset\}$, where $\emptyset$ denotes the empty word,
	$W^{I}_{m}:=I^{m}=\{\omega_{1}\ldots\omega_{m}\mid\textrm{$\omega_{k}\in I$ for any $k\in\{1,\ldots,m\}$}\}$
	for $m\in\mathbb{N}$ and $W^{I}_{*}:=\bigcup_{m\in\mathbb{N}\cup\{0\}}W^{I}_{m}$,
	which are regarded as subsets of $W_{*}$ in the natural manner.
%	For $w\in W^{I}_{*}$, the unique $m\in\mathbb{N}\cup\{0\}$ satisfying $w\in W^{I}_{m}$
%	is denoted by $|w|_{I}$ and called the \emph{$I$-length} of $w$.
\item[\textup{(2)}]We set
	$\Sigma^{I}:=I^{\mathbb{N}}=\{\omega_{1}\omega_{2}\omega_{3}\ldots\mid\textrm{$\omega_{k}\in I$ for any $k\in\mathbb{N}$}\}$,
	which is equipped with the product topology of the discrete topology on $I$,
	and define the \emph{shift map} $\sigma^{I}:\Sigma^{I}\to\Sigma^{I}$ by
	$\sigma^{I}(\omega_{1}\omega_{2}\omega_{3}\ldots):=\omega_{2}\omega_{3}\omega_{4}\ldots$.
	For $w\in W^{I}_{*}$ we define $\sigma^{I}_{w}:\Sigma^{I}\to\Sigma^{I}$ by
	$\sigma^{I}_{w}(\omega_{1}\omega_{2}\omega_{3}\ldots):=w\omega_{1}\omega_{2}\omega_{3}\ldots$.
	For $\omega=\omega_{1}\omega_{2}\omega_{3}\ldots\in\Sigma^{I}$ and
	$m\in\mathbb{N}\cup\{0\}$, we set $[\omega]^{I}_{m}:=\omega_{1}\ldots\omega_{m}\in W^{I}_{m}$.
\end{itemize}
\end{dfn}
\begin{dfn}\label{dfn:Markov-chain-AG}
Set $\widetilde{\Gamma}:=\{g\in\Gamma\mid\HM(g)=1\}$,
$[g]_{\Gamma}:=\HM(g)^{1/d}g\in\widetilde{\Gamma}$ for each $g\in\Gamma$,
$\Omega:=\widetilde{\Gamma}\times\Sigma^{I}$, which is equipped with the product
topology, and let $\mathscr{F}$ denote the Borel $\sigma$-field of $\Omega$. We define
%a continuous map
$\theta:\Omega\to\Omega$ by
$\theta(g,\omega):=\bigl([gM_{[\omega]^{I}_{1}}]_{\Gamma},\sigma^{I}(\omega)\bigr)$,
and for each $n\in\mathbb{N}\cup\{0\}$ we define
%random variables
$X_{n}:\Omega\to\widetilde{\Gamma}$ and $u_{n},\mathscr{V}_{n}:\Omega\to\mathbb{R}$
%on $(\Omega,\mathscr{F})$
by $X_{n}(g,\omega):=[gM_{[\omega]^{I}_{n}}]_{\Gamma}$,
$u_{n}(g,\omega):=d^{-1}\log\bigl(\HM\bigl(gM_{[\omega]^{I}_{n}}\bigr)/\HM\bigl(gM_{[\omega]^{I}_{n+1}}\bigr)\bigr)$
and $\mathscr{V}_{n}(g,\omega):=-d^{-1}\log\HM\bigl(gM_{[\omega]^{I}_{n}}\bigr)=\sum_{k=0}^{n-1}u_{k}$,
so that $X_{n}\circ\theta=X_{n+1}$ and $u_{n}\circ\theta=u_{n+1}$ for any $n\in\mathbb{N}\cup\{0\}$.
Also for each $g\in\widetilde{\Gamma}$, we define $\mathbb{P}_{g}$
as the unique probability measure on $(\Omega,\mathscr{F})$ such that
$\mathbb{P}_{g}[\{g\}\times\sigma^{I}_{w}(\Sigma^{I})]=\HM(gM_{w})$ for any $w\in W^{I}_{*}$,
and the expectation with respect to $\mathbb{P}_{g}$ is denoted by $\mathbb{E}_{g}[(\cdot)]$.
\end{dfn}
Clearly $\bigl(\Omega,\mathscr{F},\{X_{n}\}_{n\in\mathbb{N}\cup\{0\}},\{\mathbb{P}_{g}\}_{g\in\widetilde{\Gamma}}\bigr)$
is a time-homogeneous Markov chain on $\widetilde{\Gamma}$ with transition function
$\mathscr{P}(g,\cdot):=\mathbb{P}_{g}[X_{1}\in\cdot]=\sum_{\tau\in I}\HM(gM_{\tau})\delta_{[gM_{\tau}]_{\Gamma}}$,
where $\delta_{g}$ for $g\in\widetilde{\Gamma}$ denotes the (unique) Borel probability
measure on $\widetilde{\Gamma}$ such that $\delta_{g}(\{g\})=1$. It is also easy
to see that for each $g\in\widetilde{\Gamma}$ and each $n\in\mathbb{N}\cup\{0\}$,
\emph{the conditional law $\mathbb{P}_{g}[u_{n}\in\cdot\mid\{X_{k}\}_{k\in\mathbb{N}\cup\{0\}},\{u_{k}\}_{k\in(\mathbb{N}\cup\{0\})\setminus\{n\}}]$
of $u_{n}$ is determined solely by $X_{n},X_{n+1}$}; see \cite[(1.1)]{Kesten:AOP1974} for the precise
formulation of this property. Thus the stochastic processes of Definition \ref{dfn:Markov-chain-AG}
fall within the general framework of Kesten's renewal theory in \cite{Kesten:AOP1974}.

\newpage
The idea of considering these stochastic processes is hinted by \eqref{eq:general-counting-AG-remainder}.
Indeed, assume the setting of Theorem \ref{thm:general-counting-AG}, set
$\widetilde{\Gamma}'_{\varepsilon_{0}}:=\widetilde{\Gamma}\cap\Gamma'\cap\Gamma_{\varepsilon_{0}}$
and define
$\widetilde{\mathscr{N}},\mathscr{R}_{0}:\widetilde{\Gamma}'_{\varepsilon_{0}}\times\mathbb{R}\to[0,+\infty)$
by $\widetilde{\mathscr{N}}(g,s):=e^{-ds}\mathscr{N}(g,e^{s})$ and
$\mathscr{R}_{0}(g,s):=e^{-ds}\bigl(\mathscr{N}(g,e^{s})-\sum_{\tau\in I}\mathscr{N}(gM_{\tau},e^{s})\bigr)$.
Then it easily follows from the assumptions of Theorem \ref{thm:general-counting-AG} that
$0\leq\mathscr{R}_{0}(g,s)\leq c_{9}e^{-(d-\eta)|s|}$ for any
$(g,s)\in\widetilde{\Gamma}'_{\varepsilon_{0}}\times\mathbb{R}$
for some $c_{9}\in(0,+\infty)$ and that for any
$(g,s)\in\widetilde{\Gamma}'_{\varepsilon_{0}}\times\mathbb{R}$,
\begin{align}\notag
\widetilde{\mathscr{N}}(g&,s)-\mathscr{R}_{0}(g,s)
	=\sum\nolimits_{\tau\in I}\HM(gM_{\tau})e^{-d(s+d^{-1}\log\HM(gM_{\tau}))}
		\mathscr{N}\bigl([gM_{\tau}]_{\Gamma},e^{s+d^{-1}\log\HM(gM_{\tau})}\bigr)\\
&=\sum\nolimits_{\tau\in I}\HM(gM_{\tau})\widetilde{\mathscr{N}}\bigl([gM_{\tau}]_{\Gamma},s+d^{-1}\log\HM(gM_{\tau})\bigr)
	=\mathbb{E}_{g}\bigl[\widetilde{\mathscr{N}}(X_{1},s-\mathscr{V}_{1})\bigr].
\label{eq:general-counting-AG-pre-renewal}
\end{align}
A repetitive use of \eqref{eq:general-counting-AG-pre-renewal} further shows that
\begin{equation}\label{eq:general-counting-AG-renewal}
\widetilde{\mathscr{N}}(g,s)
	=\mathbb{E}_{g}\Bigl[\sum\nolimits_{n\in\mathbb{N}\cup\{0\}}\mathscr{R}_{0}(X_{n},s-\mathscr{V}_{n})\Bigr]
	\qquad\textrm{for any $(g,s)\in\widetilde{\Gamma}'_{\varepsilon_{0}}\times\mathbb{R}$,}
\end{equation}
which together with $\mathscr{R}_{0}(g,s)\leq c_{9}e^{-(d-\eta)|s|}$ implies that
$\sup_{(g,s)\in\widetilde{\Gamma}'_{\varepsilon_{0}}\times\mathbb{R}}\widetilde{\mathscr{N}}(g,s)<+\infty$.

Thus the proof of Theorem \ref{thm:general-counting-AG} is reduced to proving that
a function on $\widetilde{\Gamma}'_{\varepsilon_{0}}\times\mathbb{R}$ of the form
\eqref{eq:general-counting-AG-renewal} converges to $c_{8}$ as $s\to+\infty$ for any
$g\in\widetilde{\Gamma}'_{\varepsilon_{0}}$ for some $c_{8}\in(0,+\infty)$. This is exactly what
Kesten's renewal theorem \cite[Theorem 2]{Kesten:AOP1974} asserts under a reasonable set of assumptions on
$\bigl(\Omega,\mathscr{F},\{(X_{n},u_{n})\}_{n\in\mathbb{N}\cup\{0\}},\{\mathbb{P}_{g}\}_{g\in\widetilde{\Gamma}}\bigr)$
and $\mathscr{R}_{0}$, among which the following \emph{unique ergodicity}
of $\{(X_{n},u_{n})\}_{n\in\mathbb{N}\cup\{0\}}$ is the most important:
\begin{prop}\label{prop:invariant-meas}
\begin{itemize}[label=\textup{(1)},align=left,leftmargin=*,parsep=0pt,itemsep=0pt]
\item[\textup{(1)}]There exists a unique Borel probability measure $\nu$ on $\widetilde{\Gamma}$ such that
	$\nu(A)=\int_{\widetilde{\Gamma}}\mathscr{P}(g,A)\,d\nu(g)$ for any Borel subset $A$ of $\widetilde{\Gamma}$.
	Moreover, $\nu(\widetilde{\Gamma}\cap\Gamma^{\circ}\cap\Gamma_{\varepsilon_{0}})=1$.
\item[\textup{(2)}]The probability measure $\mathbb{P}_{\nu}$ on $(\Omega,\mathscr{F})$
	defined by $\mathbb{P}_{\nu}[A]:=\int_{\widetilde{\Gamma}}\mathbb{P}_{g}[A]\,d\nu(g)$
	for each $A\in\mathscr{F}$ is invariant and ergodic with respect to $\theta$.
\end{itemize}
\end{prop}
\begin{proof}
\begin{itemize}[label=\textup{(1)},align=left,leftmargin=*,parsep=0pt,itemsep=0pt]
\item[\textup{(1)}]Since
	$\mathscr{P}(g,\widetilde{\Gamma}\cap\Gamma_{\varepsilon_{0}})=1$
	for any $g\in\widetilde{\Gamma}$ by \eqref{eq:AG-HM-I-Gamma-epsilon0},
	any Borel probability measure $\nu$ on $\widetilde{\Gamma}$ with the property
	$\nu=\int_{\widetilde{\Gamma}}\mathscr{P}(g,\cdot)\,d\nu(g)$ must satisfy
	$\nu(\widetilde{\Gamma}\cap\Gamma_{\varepsilon_{0}})=1$, and the Markov chain
	$\{X_{n}\}_{n\in\mathbb{N}\cup\{0\}}$ can be restricted to the \emph{compact} set
	$\widetilde{\Gamma}\cap\Gamma_{\varepsilon_{0}}$. Then we easily see from
	Corollary \ref{cor:AG-HM-Lipschitz-I} that its transition function
	$\mathscr{P}_{0}:=\mathscr{P}|_{\widetilde{\Gamma}\cap\Gamma_{\varepsilon_{0}}}$
	has the property that $\{\mathscr{P}_{0}^{n}f\}_{n\in\mathbb{N}\cup\{0\}}$ is
	uniformly equicontinuous on $\widetilde{\Gamma}\cap\Gamma_{\varepsilon_{0}}$ for any
	continuous function $f:\widetilde{\Gamma}\cap\Gamma_{\varepsilon_{0}}\to\mathbb{R}$.
	Now we can easily conclude the existence of $\nu$ from the classical theorem
	of Krylov and Bogolioubov (see \cite[Theorem 1.10]{Hairer:convergence}),
	its uniqueness from \cite[Theorem 4.1.11]{Zaharopol:InvProb2005}, and
	$\nu(\widetilde{\Gamma}\cap\Gamma^{\circ})=1$ from the uniqueness of $\nu$ and
	the fact that $\mathscr{P}(g,\widetilde{\Gamma}\cap\Gamma^{\circ})=1$ for any
	$g\in\widetilde{\Gamma}\cap\Gamma^{\circ}$ by \eqref{eq:M1n23M2n31M3n12}.
\item[\textup{(2)}]This can be deduced from (1) by standard arguments;
	see \cite[Section 8]{K:WeylAG} for details.\hspace*{-9.8pt}\qedhere
\end{itemize}
\end{proof}
Now it is not difficult to verify \cite[Conditions I.1--I.4]{Kesten:AOP1974}
on the basis of Proposition \ref{prop:invariant-meas}, \eqref{eq:AG-HM-comparable},
\eqref{eq:AG-HM-Lipschitz-log} and \eqref{eq:AG-HM-Lipschitz-I-Mw}, which allows us to apply
\cite[Theorem 2]{Kesten:AOP1974} to the present case and thereby to get the following theorem.
\begin{thm}[{\cite[Theorem 2]{Kesten:AOP1974} applied to the present case}]\label{thm:Kesten-renewal}
Let $\Gamma'$ denote either of $\Gamma$ and $\Gamma^{\circ}$, set
$\widetilde{\Gamma}'_{\varepsilon_{0}}:=\widetilde{\Gamma}\cap\Gamma'\cap\Gamma_{\varepsilon_{0}}$,
and let $f:\widetilde{\Gamma}'_{\varepsilon_{0}}\times\mathbb{R}\to\mathbb{R}$ be continuous and satisfy
$|f(g,s)|\leq ce^{-|s|/c}$ for any $(g,s)\in\widetilde{\Gamma}'_{\varepsilon_{0}}\times\mathbb{R}$
for some $c\in(0,+\infty)$. Then for any $g\in\widetilde{\Gamma}'_{\varepsilon_{0}}$,%
\begin{equation}\label{eq:Kesten-renewal}
\lim_{s\to+\infty}\mathbb{E}_{g}\Bigl[\sum\nolimits_{n\in\mathbb{N}\cup\{0\}}f(X_{n},s-\mathscr{V}_{n})\Bigr]
	=\Bigl(\int_{\widetilde{\Gamma}}\mathbb{E}_{y}[u_{0}]\,d\nu(y)\Bigr)^{-1}
		\int_{\widetilde{\Gamma}'_{\varepsilon_{0}}}\int_{\mathbb{R}}f(y,s)\,ds\,d\nu(y).\mspace{-8.435mu}
\end{equation}
\end{thm}
Finally, we can easily deduce Theorem \ref{thm:general-counting-AG} from Theorem \ref{thm:Kesten-renewal}.
The function $\mathscr{R}_{0}$ as in \eqref{eq:general-counting-AG-renewal} is obviously
not continuous, but this problem can be avoided by considering the function
$\mathscr{Z}(g,t):=\sum_{n\in\mathbb{N}}e^{-\lambda_{n}(g)t}$ instead of
$\mathscr{N}(g,\lambda)$. Indeed, the boundedness of $\widetilde{\mathscr{N}}$
implies that of the function $\widetilde{\mathscr{Z}}(g,s):=e^{-ds}\mathscr{Z}(g,e^{-s})$,
and $\widetilde{\mathscr{Z}}$ and the continuous function
$\mathscr{R}(g,s):=e^{-ds}\bigl(\mathscr{Z}(g,e^{-s})-\sum_{\tau\in I}\mathscr{Z}(gM_{\tau},e^{-s})\bigr)$
on $\widetilde{\Gamma}'_{\varepsilon_{0}}\times\mathbb{R}$ are easily seen to satisfy
the same properties as those for $\widetilde{\mathscr{N}},\mathscr{R}_{0}$ mentioned above.
It follows that Theorem \ref{thm:Kesten-renewal} is applicable to $f=\mathscr{R}$ and
implies that $\lim_{s\to+\infty}\widetilde{\mathscr{Z}}(g,s)=c_{0}'$ for any
$g\in\widetilde{\Gamma}'_{\varepsilon_{0}}$ for some $c_{0}'\in[0,+\infty)$, but
$c_{0}'=0$ would mean $\mathscr{R}|_{\supp[\nu]\times\mathbb{R}}=0$ by
\eqref{eq:Kesten-renewal} and, together with \eqref{eq:general-counting-AG-pre-renewal} and
\eqref{eq:general-counting-AG-renewal} for $\widetilde{\mathscr{Z}},\mathscr{R}$, would easily
result in the contradiction that $\widetilde{\mathscr{Z}}|_{\supp[\nu]\times\mathbb{R}}=0$,
where $\supp[\nu]$ denotes the smallest closed subset of $\widetilde{\Gamma}'_{\varepsilon_{0}}$
with $\nu(\supp[\nu])=1$. Thus $c_{0}'\in(0,+\infty)$, and now Theorem \ref{thm:general-counting-AG}
follows by an application of Karamata's Tauberian theorem \cite[p.\ 445, Theorem 2]{Feller:VolII2nd}
to $\mathscr{N}(g,\cdot)$ for each $g\in\widetilde{\Gamma}'_{\varepsilon_{0}}$.
%
%%%% 参考文献
\vspace*{-2.02pt}%
{\footnotesize%
}%
{\small\noindent Department of Mathematics, Graduate School of Science, Kobe University\\
Rokkodai-cho 1-1, Nada-ku, 657-8501 Kobe, Japan\\
\texttt{nkajino@math.kobe-u.ac.jp}}%
%
%  %%%% 付録
%  \appendix
%  \section{付録}
\end{document}